\documentclass[a4paper, 12pt]{article}
\usepackage{stmaryrd}
\usepackage{amssymb}
\usepackage{mathrsfs}
\usepackage{latexsym}
\usepackage{amsmath}
\usepackage{amsfonts, amssymb, amscd, amsthm}
\usepackage[all]{xy}
\usepackage[dvipdfm]{graphicx}
\usepackage{color}

\usepackage[bookmarksnumbered, bookmarksopen]{hyperref}

\input amssym.def
\input amssym

\input amssym.def
\input amssym

\newcommand{\C}{\Bbb C}

\newcommand{\R}{\Bbb R}

\newcommand\rw{\rightarrow}

\newcommand\lge{\langle}
\newcommand\rg{\rangle}

\newtheorem{theorem}{Theorem}[section]
\newtheorem{proposition}[theorem]{Proposition}
\newtheorem{corollary}[theorem]{Corollary}
\newtheorem{lemma}[theorem]{Lemma}

\newtheorem{definition}[theorem]{Definition}

\begin{document}
\title{\bf Vanishing Theorems on Compact Hyperk\"ahler Manifolds
}
\author{
Qilin Yang
 \thanks{ Department of Mathematics,
Sun Yat-Sen University, ~510275, ~ Guangzhou, ~P. R.  CHINA.}
\thanks{{E-mail:~yqil@mail.sysu.edu.cn.}}}
\date{}
\maketitle

\pagenumbering{arabic}
 \setcounter{page}{1}
 \begin{center}
\begin{minipage}{13cm}
ABSTRACT. We prove that if $B$ is a $k$-positive holomorphic line
bundle on a compact hyperk\"ahler manifold $M,$ then $H^p
(M,\Omega^q\otimes B)=0$  for $p>n+[\frac{k}{2}]$  and any
nonnegative integer $q.$ In a special case $k=0$ and $q=0$ we
recover a vanishing theorem of Verbitsky's with a little stronger
 assumption.
\end{minipage}
\end{center}
\section{Introduction}
A hyperk\"ahler manifold is an oriented $4n$-dimensional  Riemannian
manifold with a special holonomy group $Sp(n)\subset SO(4n).$ The
holonomy group $U(n)\subset SO(2n)$ corresponds to a K\"ahler
manifold. The unitary group $U(n)$ is exactly the subgroup of
$SO(2n)$ that preserves a complex structure which together with a compartible
Riemannian metric defines a symplectic form. Hence a K\"ahler
manifold can also be defined as a Riemannian manifold with compatible
 symplectic structure and complex structure. By the same
reasoning, $Sp(n)\subset SU(2n)$ is a subgroup exactly preserving
three complex structures $I,J,K$ with $IJ=-JI=K.$ As the name suggests,
a hyperk\"ahler manifold is also characterized as a Riemannian
manifold with three compatible complex structures $I,J,K$ with $IJ=-JI=K,$ and a
compatible  symplectic form which is K\"ahler with respect to each one
of $I,J,K.$

A hyperk\"ahler manifold $M$ is called irreducible if $H^1(M)=0$ and
$H^{2,0}(M)=\C.$ By Bogomolov's decomposition theorem for a K\"ahler
manifold with trivial canonical class \cite{bo1},\cite{bea}, up to
finite cover, any hyperk\"ahler manifold  is biholomorphic to a
product of irreducible hyperk\"ahler manifolds and a hyperk\"ahler
complex torus. On any hyperk\"ahler manifold $M$ there is a symmetric bilinear form $q,$ called
{\it Beauville-Bogomolov-Fujiki form}, which takes positive values on the K\"ahler cone $\mathscr{K}$ of $(M,I).$
The closure  of the  dual K\"ahler cone
 $\mathscr{K}^{\vee}=:\{x\in H^{1,1}(M,\R)|q(x,y)>0,\forall y\in \mathscr{K}\}$ is denoted by
  $\overline{\mathscr{K}^{\vee}}.$
In \cite{ve1}, Verbitsky established the following
vanishing theorem for a compact irreducible hyperk\"ahler manifold:

\begin{theorem}{\rm ({\bf Verbitsky, 2007},\cite{ve1})}\label{vbts} Let $M$ be a compact irreducible hyperk\"ahler
manifold of real dimension $4n$, and let $L$ be a holomorphic line
bundle on $M.$ If   $0\not=c_1(L)\in \overline{\mathscr{K}^{\vee}},$
in particular, if $L$ is a positive line bundle, then
\begin{equation}H^p(M,L)=0,~~~~~{\mbox for}~~p>n.\end{equation}
\end{theorem}

Verbitsky's proof of the theorem above is a clever use of the symmetric
pole of the complex structures and the holomorphic Bochner-Kodaira-Nakano
type identity, which appeared already in \cite{fu}. In the proof the
assumption of irreducibility was used in an essential way. In this
paper we use a different method, to establish some vanishing theorems for more general hypercomplex K\"ahler
manifolds. To get the flavor, we state the
following  result:

 \begin{theorem}\label{maincor} Let $M$ be a compact hyperk\"ahler
manifold of real dimension $4n$, and let $L$ be a holomorphic line
bundle on $(M,I).$ If $L$ is a positive line bundle, then for any
$q\geq 0,$
\begin{equation} H^p(M,\Omega^q\otimes L)=0,~~~{\mbox {for}}~~~p>n.\end{equation}
\end{theorem}

Note that we don't assume $M$ is irreducible.  We recover Theorem
\ref{vbts} if $q=0$ while a stronger
 assumption. During the proof of Theorem \ref{maincor}, the K\"ahler metric of $(M,I)$ is changed, the new
K\"ahler metric is not necessarily hyperk\"ahler. So we develop our theory on hypercomplex K\"ahler manifolds,
and deal with hyperk\"ahler manifolds as their special examples. After deriving the Bochner-Kodaira-Nakano
identities from Section 2 to Section 4, we get our main results finally in Section 5.

\section{Preliminary}
\setcounter{section}{2}
\setcounter{equation}{0}
A Hermitian manifold $M$ is a complex manifold with an integrable complex structure $I$ and
a Riemannian metric $g$ satisfying the compatible  condition $g(IX,IY)=g(X,Y)$ for any $X,Y\in TM.$
It is called a K\"ahler manifold if in addition the 2-form
$\omega$ defined by $\omega(X,Y)=-g(X,IY)$ is  symplectic form. $\omega$ is also called the
{\it K\"ahler form}
associated with  $g$ and $g$ is called a {\it K\"ahler metric.} On the other hand, if we start with
 a symplectic manifold $M$ equipped
with a symplectic form $\omega$, then  $M$ is a K\"ahler manifold
if and only if there is a $\omega$-compatible
 integrable complex structure. Recall that a complex structure $I$ is called {\it $\omega$-compatible}
 if $\omega$ is {\it $I$-invariant} $\omega(IX,IY)=\omega(IX,IY)$for any $X,Y\in TM,$ and {\it $I$-tamed} $\omega(X,IX)>0$ if $X\not=0.$
  To see that two
 definitions of K\"ahler manifolds are equivalent it suffices to note that
 $g(X,Y)=\omega(X,IY)$ is a Hermitian metric in the latter definition of K\"ahler manifolds.

\begin{definition} A Riemannian manifold $M$ with  Riemannian metric $g$
is called a hyperk\"ahler manifold if it admits three integrable complex
structures $I,J,K$ with $IJ=-JI=K$ such that $g$ is a K\"ahler metric with respect to each one of $I,J,K.$ We
called $g$ a hyperk\"ahler metric.
\end{definition}

\begin{proposition} Let $(M,\omega,I)$ be a K\"ahler manifold  with Levi-Civita
connection $\nabla.$ Then $M$
is hyperk\"ahler if there are integrable complex structures $J,K$ with $IJ=-JI=K$ and $J,K$
 satisfying the following conditions:

{\rm (i).} $\omega(JX,KY)=\omega(X,IY);$

{\rm (ii).} $J,K$ are parallel: $\nabla J=\nabla K=0.$
\end{proposition}
\begin{proof} Let $g(X,Y)=\omega(X,IY),$ then $g$ is a K\"ahler metric by the assumptions.
  $g$ is Hermitian
relative to $J$ if
 $g(X,Y)=g(JX,JY),$ which is $\omega(JX,KY)=\omega(X,IY).$
 Let $\omega_J (X,Y)=-g(X,JY)=-\omega(X,KY).$
Clearly $\omega_J$ is nondegenerate. It is well-known that $d\omega_J=0$ if $\nabla J=0.$ Hence $g$
is K\"ahler with respect to $J$ if (i) and (ii) are true. $g$
is K\"ahler with respect to $K$ follows in the same way if (i) and (ii) are true.
\end{proof}

\begin{proposition}\label{ebas} Let $M$ be a complex symplectic manifold with an
integrable complex structure $I$ and an $I$-invariant symplectic
structure $\omega.$ Let $g(X,Y)=\omega(X,IY)$  for any $X,Y\in TM.$
Then $g$ is a Lorentz Hermitian metric and for any point $x\in M$
there exists a local holomorphic coordinate $(w^1,\cdots,w^n)$
around $x$ such that
$$g=\sum_{jk}(\pm \delta_{jk}+O(|w|^2))dw^jd\bar{w}^k.$$
\end{proposition}
\begin{proof}
If moreover $I$ is $\omega$-partible then $M$ is a K\"aler manifold and there is a proof in \cite{gr} for this special case.
The general case in our proposition follows in the same way.
Since $\omega$ is a symplectic form, $g$
is nondegenerate but not necessarily positive definite, hence a Lorentz
Hermitian metric.
We could find local holomorphic coordinates
$(z^j)$ at $x$ such that $g_{jk}(x)=\pm \delta_{jk},$
in other words  we could write locally
$$\begin{array}{rcl}
g &=&\sum_{jkl}(\pm\delta_{jk}+a_{jkl}z_l +a_{jk\bar{l}}\bar{z}^l +O(|z|^2 ))dz^j d\bar{z}^k;\\
\omega &=&i \sum_{jkl}(\pm\delta_{jk}+a_{jkl}z_l
+a_{jk\bar{l}}\bar{z}^l +O(|z|^2 ))dz^j\wedge d\bar{z}^k.
\end{array}$$
Let us make a holomorphic change of coordinates
$$z_k=w_k +\frac{1}{2}\sum b_{klm}w_lw_m,$$
Then in the new coordinate we have
$$g=\sum_{jk}(\pm\delta_{jk}+\sum_l (a_{jkl}w_l +a_{jk\bar{l}}\bar{w}_l+
b_{klj}w_l +\overline{b_{jlk}}\bar{w}_l +O(|w|^2 )))dw^j d\bar{w}^k.$$
Choose $b_{klj}=-a_{jkl}.$ Since $g$ and $\omega$ are real, we have
\begin{equation}\label{a}\overline{a_{jkl}}=a_{jk\bar{l}}.\end{equation}
Since $d\omega=0,$ in particular $d\omega(x)=0,$ hence
\begin{equation}\label{b} a_{jkl}=a_{ljk}.\end{equation}
From (\ref{a}) and (\ref{b}) we have
$$\overline{b_{jlk}}=-\overline{a_{kjl}}=-a_{jk\bar{l}},$$
therefore, locally $g=\sum_{jk}(\pm \delta_{jk}+O(|w|^2))dw_jd\bar{w}_k.$
\end{proof}

A complex manifold $M$ with integrable complex structures $I,J,K$ is called a hypercomplex manifold
if $IJ=-JI=K,$ and $(I,J,K)$ is called a {\it hypercomplex structure} (Verbitsky has studied hypercomplex manifolds and
 hypercomplex K\"ahler manifolds in a series papers \cite{ve2},\cite{ve1},\cite{av}).
Obata proved that on a hypercomplex manifold $(M,I,J,K)$ there exists
a unique  torsion-free
connection such that $I,J,K$ are parallel \cite{ob}:
$$\nabla I = \nabla J = \nabla K = 0.$$
Such a connection is called an {\it Obata connection.} If $M$ is a hyperk\"ahler manifold,
 clearly  the Levi-Civita
connection is exactly the Obata connection of the  underlying
hypercomplex manifold.

The following proposition is cited from \cite{av}.

\begin{proposition}\label{app} Let $(M,I,J,K)$ be a hypercomplex manifold. For any point $x\in M,$
 there exists a holomorphic (with respect to $I$) local  coordinate $(z^j)$  around $x$ with $z^j(x)=0$ such that
 $$\begin{array}{rcl}
I(z) &=& I_0 + O(|z|^2)I';\\
J(z) &=& J_0 + O(|z|^2)J';\\
K(z) &=& K_0 + O(|z|^2)K';
\end{array}$$
where $I_0, J_0,K_0$ are the constant complex structures.
\end{proposition}
\begin{proof}
We could choose a normal coordinate $(z^j)$ at $x$ with $z^j(x)=0$
for the Obata connection $\nabla,$ then the Christoffel symbols of
$\nabla$ vanish at $x.$ Since at $x$ $$\nabla I = dI=\nabla J=dJ =
\nabla K =dK= 0,$$ $I(z) = I_0 + O(|z|^2)I'$ with $I_0=I(0)$ for
example.
\end{proof}

\section{Bochner-Kodaira-Nakano type identities}

\subsection{Generalized Hodge identities for differential forms}
\setcounter{section}{3}
 \setcounter{subsection}{1}
\setcounter{equation}{0}
Let $(M,I,J,K)$ be a compact hypercomplex manifold of real dimension $4n$ and $(M,I)$ is a K\"ahler
 manifold with K\"ahler metric $g$ \cite{ve2},\cite{ve1}.
There are naturally associated three nondegenerate $2$-forms
$$\begin{array}{rcl} && \omega_I(\cdot,\cdot)=g(\cdot,I\cdot),\\
&& \omega_J(\cdot,\cdot)=g(\cdot,J\cdot),\\
&& \omega_K(\cdot,\cdot)=g(\cdot,K\cdot).\\
\end{array}$$
Let
$\{\xi^1,I\xi^1,J\xi^1,K\xi^1,\cdots,\xi^n,I\xi^n,J\xi^n,K\xi^n\}$ be a real   unit
orthogonal coframe of the cotangential bundle $T^*M$ at a fixed point $x\in M.$
Then using the Darboux theorem, we can write the K\"ahler form $\omega_I$ locally as
 $$\omega_I=\sum_{j=1}^n (\xi^j\wedge I\xi^j +J\xi^j\wedge K\xi^j),$$
Accordingly,   $$\begin{array}{rcl}
&& \omega_J=\sum_{j=1}^n (\xi^j\wedge J\xi^j +K\xi^j\wedge I\xi^j),\\
&& \omega_K=\sum_{j=1}^n (\xi^j\wedge K\xi^j +I\xi^j\wedge J\xi^j).\\
\end{array}$$
Choose holomorphic coframes relative to the complex structure $I,$
\begin{equation}\label{holc}\theta^j=\xi^j-iI\xi^j,\theta^{j+n}=J\xi^{j}-iK\xi^{j},~~~j=1,\cdots,n,\end{equation}
with antiholomorphic coframes
$$\bar{\theta}^j=\xi^j+iI\xi^j,\bar{\theta}^{j+n}=J\xi^{j}+iK\xi^{j},~~~j=1,\cdots,n.$$
Then  we have the following pointwise action at $x\in M:$
\begin{align}
I\theta^j=i\theta^j,~~~~~& I\bar{\theta}^j=-i~\bar{\theta}^j,& j=1,\cdots,2n;\label{i}\\
J\theta^j=\bar{\theta}^{j+n},~~~~~&J\theta^{j+n}=-\bar{\theta}^{j},&j=1,\cdots,n;\label{l}\\
K\theta^j=-i\bar{\theta}^{j+n},~~~~~&
K{\theta}^{j+n}=i\bar{\theta}^j,&j=1,\cdots,n.\label{k}
\end{align}
Using holomorphic and antiholomorphic coframes, we could rewrite
 the 2-forms $\omega_I,\omega_J,\omega_K$ locally as
$$\begin{array}{rcl} && \omega_I=\frac{i}{2}\sum_{j=1}^n (\theta^j\wedge \bar{\theta}^j
 +\theta^{j+n}\wedge \bar{\theta}^{j+n}),\\
&& \omega_J=\frac{1}{2}\sum_{j=1}^n (\theta^j\wedge {\theta}^{j+n}
 +\bar{\theta}^{j}\wedge \bar{\theta}^{j+n}),\\
&& \omega_K=\frac{i}{2}\sum_{j=1}^n (\bar{\theta}^j\wedge
\bar{\theta}^{j+n}
 -{\theta}^{j}\wedge {\theta}^{j+n}).\\
\end{array}$$
Thus $$\varphi=\omega_J +i\omega_K=\sum_{j=1}^n \theta^j\wedge
{\theta}^{j+n}$$ is a holomorphic $(2,0)$-form relative the complex structure $I.$

For convenience,  when talking about holomorphic structure of $M$
we always mean that it is relative to the complex structure $I$ if without special mention
 in the rest of this paper, though $I, J$ and $K$ have symmetric and equal positions.

\begin{theorem} \label{imp} There exists a local holomorphic coordinate $(z^1,\cdots,z^n)$ around $x$ such that
\begin{align}\omega_I =&\sum_{jk}( \delta_{jk}+O(|z|^2))dz^j\wedge dz^k;\label{3.5}\\
\omega_J =&\sum_{jk}( \delta_{jk}+O(|z|^2))(dz^j\wedge dz^{k+n}+d\bar{z}^j\wedge d\bar{z}^{k+n});\label{3.6}\\
\omega_K =&\sum_{jk}( \delta_{jk}+O(|z|^2))(dz^j\wedge dz^{k+n}-d\bar{z}^j\wedge d\bar{z}^{k+n}).\label{3.7}
\end{align}

\end{theorem}
\begin{proof}
From \cite{nn}, we know there exists a local holomorphic coordinates
$(z^j)$ with respect to the complex structure $I$ such that its
action is local constant: $Idz^j=idz^j$ with $\theta^j=dz^j,$ which
coincides with the pointwise action we considered in (\ref{i}). By
Proposition \ref{ebas}, there exists a local holomorphic coordinates
$(z^j)$ such that (\ref{3.5}) holds. By Proposition \ref{app} and
(\ref{l}), (\ref{k}),
\begin{align}
Jdz^j =dz^{j+n}+O(|z|^2)\sum_{k}( dz^k +d\bar{z}^k);\label{ap1}\\
Jdz^{j+n} =-dz^{j}+O(|z|^2)\sum_{k}( dz^k +d\bar{z}^k);\\
Kdz^j =idz^{j+n}+O(|z|^2)\sum_{k}( dz^k +d\bar{z}^k);\\
Kdz^{j+n} =-idz^{j}+O(|z|^2)\sum_{k}( dz^k +d\bar{z}^k).\label{ap4}
\end{align}
From (\ref{ap1}) to (\ref{ap4}) and the definition of $\omega_J$ and $\omega_K,$ we conclude (\ref{3.6})
and (\ref{3.7}).
\end{proof}

Let $d:\Omega^p(M)\rw \Omega^{p+1}(M)$ be the de Rham differential
operator on $M,$ and let $d^c =I^{-1}dI.$ Note that the complex
structures $I,J,K$ on the tangent bundle naturally induces operator
actions on the vector fields and the differential forms. Take $I$
for an example. For $\alpha,\cdots,\beta\in \Omega^{\bullet}(M),$
the action of $I$ on differential forms are defined by
\begin{equation}\label{d1}
I(\alpha\wedge\cdots\wedge\beta)=I\alpha\wedge\cdots\wedge
I\beta.\end{equation}
The Dolbeault operators $\partial, \bar{\partial}$ and $d,d^c$ are related by
\begin{equation}\label{d2}
d=\partial +\bar{\partial },~~~d^c=-i(\partial -\bar{\partial }).
\end{equation}
Clearly the Dolbeault operators $\partial, \bar{\partial}$ are determined
completely by $d$ and the complex structure $I:$
\begin{equation}\label{d3}
\partial=\frac{1}{2}(d+id^c);~~~\bar{\partial}=\frac{1}{2}(d-id^c).
\end{equation}
Accordingly, the complex structure $J$ and $K$ also induce complex
differential operators
\begin{equation}\label{d4}
d_J =J^{-1}dJ =\partial_J +\bar{\partial }_J,~~~ d_K =K^{-1}dK
=\partial_K +\bar{\partial }_K,\end{equation} where
$\partial_J=J^{-1}\partial J$ and $\bar{\partial}_J,{\partial}_K,
\bar{\partial}_K$ are similar notions. Like
$d,\partial,\bar{\partial},$ it is easy to check that, the operators
$d_J,d_K,\partial_J,\partial_K$ also satisfy the graded Leibniz
rule: for any $\xi,\eta\in \Omega^{\bullet}(M),$
\begin{equation}\label{d5}
d_J(\xi\wedge\eta)=d_J\xi\wedge\eta+(-1)^{|\xi|}\xi\wedge
d_J\eta,\end{equation}  where $|\xi|$ is the degree of $\xi.$

For each ordered set of indices $A=\{\alpha_1,\cdots,\alpha_p\},$
denote the index length $|A|=p;$  we write
$$\theta^{A}=\theta^{{\alpha}_1}\wedge\cdots\wedge\theta^{\alpha_p},\quad,
\bar{\theta}^{A}=\bar{\theta}^{{\alpha}_1}\wedge\cdots\wedge\bar{\theta}^{{\alpha}_p},$$
and denote by $\hat{A}=\{\alpha_{p+1},\cdots,\alpha_n\}$ the
complementary of $A$ so that
$$\theta^{A}\wedge \theta^{\hat{A}}=(-1)^{\tau (A)}\theta^1\wedge\cdots\wedge \theta^{2n},$$
where $\tau (A)$ takes value $1$ if $A$ is an even permutation and
$-1$ otherwise. The {\it Hodge star operator} $\ast$ is given by
$$\ast (f\theta^{A}\wedge\overline{\theta^{B}})=2^{|A|+|B|-2n}(-1)^n\epsilon_{AB}\bar{f}
\theta^{\hat{A}}\wedge\overline{\theta^{\hat{B}}}$$ where $f$ is a
function and the signature factor
$$\epsilon_{AB} =(-1)^{n(2n-1)+(2n-p)q+\tau(A)+\tau(B)}.$$
Given two $(p,q)$-forms
$$\xi=\frac{1}{p!q!}\sum_{A,B}\xi_{A\bar{B}}\theta^A \wedge\bar{\theta}^B,~~{\rm and}~~
\eta=\frac{1}{p!q!}\sum_{A,B}\eta_{A\bar{B}}\theta^A
\wedge\bar{\theta}^B,$$ their pointwise inner product is defined by
\begin{equation}\label{pinn}
\lge\xi,\eta\rg=\xi\wedge\ast\eta=\Big(\frac{1}{p!q!} \sum_{A,B}
\xi_{A\bar{B}}\bar{\eta}_{A\bar{B}}\Big)\frac{1}{(2n)!}\omega_I^{2n}.\end{equation}
Since $M$ is compact, we can consider the Hermitian inner product on
each $\Omega^{p,q}(M)$ defined by \begin{equation}\label{inn}(\xi,
\eta)=\int_{M}\lge\xi,\eta\rg,\quad\quad \xi,\eta\in
\Omega^{p,q}(M).\end{equation}

 For each $k=1,\cdots,2n,$ let $e_k:\Omega^{p,q}(M)\rw
\Omega^{p+1,q}(M)$ be the wedge operator defined by
\begin{equation}\label{www1} e_k(\eta)=\theta_k\wedge \eta;\end{equation}
and  $\bar{e}_k:\Omega^{p,q}(M)\rw \Omega^{p,q+1}(M)$ are similarly
given by \begin{equation}\label{www2}
\bar{e}_k(\eta)=\bar{\theta}_k\wedge \eta.\end{equation} Let $i_k$
and $\bar{i}_k$ be the adjoints of $e_k$ and $\bar{e}_k$ with
respect to the inner product (\ref{inn}), respectively. They are
called contraction operators. Then for any $k,l=1,\cdots 2n,$
\begin{equation}\label{kd0}
e_k\bar{i}_l+\bar{i}_l e_k=0;
\end{equation}
\begin{equation}\label{kd3} e_k {i}_k+ {i}_k e_k=2.
\end{equation}
If $k\not=l,$
\begin{equation}\label{kd2}
e_k{i}_l+{i}_l e_k=0.
\end{equation}
The three equation above reflect how to commute the actions of
wedges and contractions. The following Proposition \ref{pro1} gives
the commutation relations between contract actions and complex
structure actions on differential forms.  Base on them, it is easy to
get the commutation relations between the actions of wedges and
complex structures
\begin{proposition}\label{pro1} As operators on acting on the differential forms,
the contractions $i_k$ and the complex structures $I,J,K$ satisfying
the commuting relations
\begin{align}i_kI=-iI{i}_{k},&\hskip1cm \bar{i}_{k}I=iI\bar{i}_{k};\label{pr11}\\
      i_kJ=J\bar{i}_{k+n},&\hskip1cm i_{k+n}J=-J\bar{i}_k;\\
      {i}_k K=iK\bar{i}_{k+n},&\hskip1cm {i}_{k+n}K=-iK\bar{i}_k.
      \end{align}
\end{proposition}
By definition equation (\ref{d4}) and the commutation relations in
Proposition \ref{pro1},  we have the following expressions of
differential operators via contraction and wedge operators:
\begin{align}
\partial=\sum_k (e_k\partial_k + e_{k+n}\partial_{k+n}),&\hskip0.5cm
{\partial}^*=-\sum_k (\bar{\partial}_k {i}_k +
\bar{\partial}_{k+n}{i}_{k+n});\label{pa1}\\
\partial_J=\sum_k (\bar{e}_{k+n}\partial_k -
\bar{e}_{k}\partial_{k+n}),&\hskip0.5cm \partial^*_J=\sum_k (-\bar{\partial}_k
\bar{i}_{k+n} +\bar{\partial}_{k+n}\bar{i}_{k});\label{pa2}\\
\partial_K=i\sum_k (\bar{e}_{k+n}\partial_k
-\bar{e}_{k}\partial_{k+n}),&\hskip0.5cm \partial^*_K=i\sum_k
(-\bar{\partial}_{k+n}\bar{i}_{k}+\bar{\partial}_k \bar{i}_{k+n}
).\label{pa3}
\end{align}
Let $L=L_I,L_J,L_K$ be the operators from $\Omega^{*}(E)$ to
$\Omega^{*}(E)$ defined by the wedge with the 2-forms
$\omega=\omega_I, \omega_J, \omega_K$ respectively and
$\Lambda=\Lambda_I, \Lambda_J,\Lambda_K$ their adjoint operators.
\begin{align}
L=\frac{i}{2}\sum_k ( e_k e_{k}+\bar{e}_{k+n}\bar{e}_{k+n}),&\hskip0.5cm
\Lambda=-\frac{i}{2}\sum_k ( i_ki_{k}+\bar{i}_{k+n}\bar{i}_{k+n});\label{w1}\\
L_J=\frac{1}{2}\sum_k (e_k e_{k+n}+\bar{e}_k\bar{e}_{k+n}),
&\hskip0.5cm \Lambda_J=\frac{1}{2}\sum_k ( i_k i_{k+n}+\bar{i}_k\bar{i}_{k+n});\label{lamj}\\
L_K=\frac{i}{2}\sum_k (\bar{e}_k\bar{e}_{k+n}- e_k e_{k+n}),&\hskip0.5cm
\Lambda_K= \frac{i}{2}\sum_k ( i_k i_{k+n}-\bar{i}_k\bar{i}_{k+n}).
\end{align}

The following identities in Lemma \ref{lem1} called {\it Hodge
identities}\cite{gr}, they  play fundamental roles in K\"ahler
geometry. Their proof are reduced from an arbitrary K\"ahler
manifold to the Euclidean K\"ahler plane via using Proposition
\ref{ebas}. The  main observation is that any intrinsically defined
identity that involves the K\"ahler metric  together with its first
derivatives and which is valid for the Euclidean metric, is also
valid on a K\"ahler manifold, since by Proposition \ref{ebas}, a
K\"ahler metric is oscalate order 2 to the Euclidean metric
everywhere.

 \begin{lemma}\label{lem1}{\rm (\bf Hodge Identities)}:
\begin{equation}\label{kd1}
[\Lambda,{\partial}]=i\bar{\partial}^*,\hskip1cm  [\Lambda,\bar{\partial}]=-i\partial^*.
\end{equation}
\end{lemma}
\noindent For a proof of this lemma, please refer to \cite{gr},
pp111-114.
\begin{proposition}\label{pro2} Let $(M,I,J,K)$ be a compact hypercomplex manifold such that $(M,I)$ is a K\"ahler manifold, then
\begin{align}
[\Lambda_J,\partial]=\bar{\partial}^*_{J},
&\hskip 1cm [\Lambda_J,\bar{\partial}]={\partial}^*_{J};\label{pr21}\\
      [\Lambda_K,\partial]=-\bar{\partial}^*_{K},& \hskip 1cm
      [\Lambda_K,\bar{\partial}]=-{\partial}^*_{K};\label{pr22}\\
      [\Lambda_I,\partial_J]=i\bar{\partial}_J^*,& \hskip1cm [\Lambda_I,\bar{\partial}_J]=-i{\partial_J}^*;\\
[\Lambda_J,\partial_J]=-\bar{\partial}^*, & \hskip1cm  [\Lambda_J,\bar{\partial}_J]=-{\partial}^*;\\
[\Lambda_K,\partial_J]=i\bar{\partial}^*,& \hskip1cm  [\Lambda_K,\bar{\partial}_J]=-i{\partial}^*;\\
 [\Lambda_I,\partial_K]=i\bar{\partial}_K^*,& \hskip1cm  [\Lambda_I,\bar{\partial}_K]=-i{\partial}_K^*;\\
[\Lambda_K,\partial_K]=\bar{\partial}^*, & \hskip1cm  [\Lambda_K,\bar{\partial}_K]={\partial}^*;\\
[\Lambda_J,\partial_K]=-i\bar{\partial}^*,& \hskip1cm  [\Lambda_J,\bar{\partial}_K]=i{\partial}^*.
\end{align}
\end{proposition}
\begin{proof}
The idea of the proof is the same with that of Lemma \ref{lem1}, since by Theorem \ref{imp}, the 2-forms
 $\omega_J$ and $\omega_K$ are oscalate order 2 to the constant 2-forms
 everywhere,  the proof reduces to the Euclidean K\"ahler plane.
 We follow the same lines of the
 proof of Lemma \ref{lem1} as in \cite{gr}, pp111-114.

Note every equation in the right column follows by taking conjugate
of the equation in the same row of  the left column, so it suffices
to establish the equations in one column. Here we only give a proof of
the left equation of (\ref{pr21}).  The rest equations are proved in the same way.

By (\ref{w1}), we have
$$\begin{array}{rcl}
2[\Lambda_J,\partial]&=&\sum_{k,l}[i_k i_{k+n}
+\bar{i}_k\bar{i}_{k+n},\partial_l e_l]\\
&=&\sum_{k}[i_k i_{k+n}
+\bar{i}_k\bar{i}_{k+n},\partial_k e_k+\partial_{k+n} e_{k+n}]+\sum_{l\not=k,k+n}[i_k i_{k+n}
+\bar{i}_k\bar{i}_{k+n},\partial_l e_l],\\
\end{array}$$
Note $$[\bar{i}_k\bar{i}_{k+n},\partial_k e_k]=[\bar{i}_k\bar{i}_{k+n},\partial_{k+n} e_{k+n}]=0,$$
and if $l\not=k,k+n,$ we have
 $$[i_k i_{k+n},\partial_l e_l]=[\bar{i}_k\bar{i}_{k+n},\partial_l e_l]=0.$$
Therefore
\begin{equation}
2[\Lambda_J,\partial]=\sum_{k}([i_k i_{k+n},\partial_k e_k]+[i_k i_{k+n},\partial_{k+n} e_{k+n}]).\label{g11}
\end{equation}
It is not difficult to check that
\begin{align}
[i_k i_{k+n}, e_k]&=-2i_{k+n},\\
[i_k i_{k+n}, e_{k+n}]&= 2i_k,
\end{align}
hence
\begin{align}
[i_k i_{k+n},\partial_k e_k]&=\partial_k[i_k i_{k+n}, e_k]=-2\partial_k i_{k+n},\label{g12}\\
[i_k i_{k+n},\partial_{k+n} e_{k+n}]&=\partial_{k+n} [i_k i_{k+n}, e_{k+n}] =2\partial_{k+n} i_k.\label{g13}
\end{align}
From (\ref{g11}),(\ref{g12}) and (\ref{g13}) we get
\begin{equation}[\Lambda_J,\partial]=\sum_{k}(-\partial_k i_{k+n}
+\partial_{k+n} i_k).\label{last}\end{equation} Note the right hand side of (\ref{last})
 is the conjugation of the second equation
of (\ref{pa2}), we arrive at the first equation of (\ref{pr21}).
\end{proof}

\subsection{Twisted Bochner-Kodaira-Nakano type identities}
In this section we will extend the differential operators studied in Sect. 3.1 to act on the
bundle-valued differential forms.

Suppose that $(M,I,J,K)$ is a compact hypercomplex manifold  and $(M,I)$ is a K\"ahler
 manifold with K\"ahler metric $g.$
Given a holomorphic vector bundle $E$ over $M$ with
Hermitian metric $h,$ there exists a unique connection $D$, called
{\it Chern connection}, which is compatible with the metric $h$ and
satisfying $D''=\bar{\partial}.$ Here $D=D'+D''$ and
$$D':\Omega^{p,q}(E)\rw \Omega^{p+1,q}(E),\quad\quad
D'':\Omega^{p,q}(E)\rw \Omega^{p,q+1}(E)$$ are its components.

 Let $\{s_{\alpha}\}$ be a
holomorphic frame of $E.$ For any $E$-valued
differential forms $\xi=\sum_{\alpha}\xi^{\alpha}\otimes s_{\alpha}$ and
$\eta=\sum_{\beta}\eta^{\alpha} \otimes s_{\alpha}$ of $\Omega^{p,q}(E),$ we
define their local inner product
$$\langle \xi,\eta\rangle=\sum_{\alpha,\beta}h_{\alpha\bar{\beta}}\xi^{\alpha}\wedge\ast \eta^{\beta}$$
and global inner product
\begin{equation}(\xi,\eta)_E=\int_M \langle \xi,\eta\rangle.\label{vinn}\end{equation}
Denote the adjoint operators of $D',D''$ with respect to the inner
product  by $\delta',\delta''.$ Let $D'_J:\Omega^{p,q}(E)\rw
\Omega^{p,q+1}(E)$ be the composition of  $$J^{-1}\otimes
id_E:E\otimes \Omega^{p,q}(M)\rw E\otimes \Omega^{q,p}(M),$$ $D'$
and $J\otimes id_E;$  and denote the adjoint operator of $D'_J$ with
respect to $(~~,~~)_E$ by $\delta'_J.$ The operators $D''_J,
\delta''_J, D'_K, \delta'_K, D''_K, \delta''_K$ are defined in the
same way.
 The operators $\ast,d,\Lambda,L,\Lambda_J,L_J,\Lambda_K,L_K$  extend
naturally to $\Omega^{p,q}(E).$ The proof of the following commuting
relations among  these operators acting on $\Omega^{p,q}(E)$ follow
from Proposition \ref{pro2}.
\begin{proposition}\label{proba}
\begin{align}
[\Lambda,D']=i\delta'',& \hskip2cm            [\Lambda,D'']=-i\delta';\\
[\Lambda_J,D']=\delta''_{J},& \hskip2cm          [\Lambda_J,D'']=\delta'_{J};\\
      [\Lambda_K,D']=-\delta''_{K},&\hskip2cm           [\Lambda_K,D'']=-\delta'_{K};\\
[\Lambda_I,D'_J]=i\delta''_J,&\hskip2cm            [\Lambda_I,D''_J]=-i\delta'_J;\\
[\Lambda_J,D'_J]=-\delta'', &\hskip2cm           [\Lambda_J,D''_J]=-\delta';\\
[\Lambda_K,D'_J]=i\delta'',&\hskip2cm             [\Lambda_K,D''_J]=-i\delta';\\
  [\Lambda_I,D'_K]=i\delta''_K,&\hskip2cm           [\Lambda_I,D''_K]=-i\delta'_K;\\
[\Lambda_K,D'_K]=\delta'', &\hskip2cm            [\Lambda_K,D''_K]=\delta';\\
[\Lambda_J,D'_K]=-i\delta'',&\hskip2cm           [\Lambda_J,D''_K]=i\delta'.
\end{align}
\end{proposition}
The {\it Chern curvature} tensor $\Theta$ of $E$ is an $End(E)$-valued
differential form
defined by
\begin{equation}
(D''D'+D'D'')\xi=\Theta\wedge\xi=e(\Theta)\xi \in\Omega^{p+1,q+1}(E),\quad\quad \xi\in\Omega^{p,q}(E).
\end{equation}
The holomorphic Laplacian $\triangle'=D'\delta'
 +\delta' D'$ and antiholomorphic Laplacian $\triangle''=D''\delta''
 +\delta'' D''$ are related by the classical {\it Bochner-Kodaira-Nakano} identity \cite{de}
 \begin{equation}\label{bkn}
\triangle''-\triangle'=[e(i\Theta),\Lambda],
\end{equation}
it plays fundamental role in establishing many important vanishing
theorems. Our naive motivation is to get more vanishing theorems by
using the chances given by the other two complex structures $J$ and
$K$ for a hypercomplex manifold. To this aim  we define the
following  self-adjoint operators
\begin{equation}\label{bknt}\triangle'_J=D'_J\delta'_J
 +\delta'_J D'_J,\quad
\triangle'_K=D'_K\delta'_K
 +\delta'_K D'_K.
\end{equation}
Let $\Theta_J, \Theta_K$ be the curvature components
corresponding to the operators $D''D'_J+D'_JD''$ and
$D''D'_K+D'_KD''$
respectively.
We have the following twisted Bochner-Kodaira-Nakano type identities.
\begin{proposition}\label{pro7}
\begin{align}
      \triangle''-\triangle'_J=[e(\Theta_J),\Lambda_J],\quad
 \triangle''-\triangle'_K=-[e(\Theta_K),\Lambda_K].
\end{align}
\end{proposition}
\begin{proof}
By Proposition \ref{proba},
$$\begin{array}{rcl}\triangle''-\triangle'_J&=&D''\delta''+\delta''D''
-(D'_J\delta'_J+\delta'_J D'_J)\\
 &=&-D''
[\Lambda_J,D'_J]-[\Lambda_J,D'_J]D''
-D'_J[\Lambda_J,D'']-[\Lambda_J,D'']
D'_J\\
&=&-D''
\Lambda_JD'_J+D''D'_J\Lambda_J
-\Lambda_JD'_JD''+D'_J\Lambda_JD''\\
&&-D'_J\Lambda_JD''+D'_JD''\Lambda_J
-\Lambda_JD''D'_J+D''\Lambda_JD'_J\\
&=&(D''D'_J+D'_JD'')\Lambda_J-\Lambda_J(D''D'_J+D'_JD'')\\
&=&[e(\Theta_J),\Lambda_J].
\end{array}$$
The second equation follows in the same way.
\end{proof}

Recall that $\varphi=\omega_J +i\omega_K,$ correspondingly, we
define $L_{{\varphi}}=L_{J}+iL_K, D'_{\bar{\varphi}}=\frac{1}{2}(D'_J-iD'_K)$ and
$ D''_{\bar{\varphi}}=\frac{1}{2}(D''_J-iD''_K).$ Then
$\Lambda_{{\varphi}}=\Lambda_J-i\Lambda_K,
\delta'_{\bar{\varphi}}=\frac{1}{2}(\delta'_J+i\delta'_K)$
and $\delta''_{\bar{\varphi}}=\frac{1}{2}(\delta''_J+i\delta''_K)$ are their
adjoint operators respectively.
\begin{proposition}\label{pro8}
\begin{align}
[\Lambda_{{\varphi}},D']=\delta''_{\bar{\varphi}},~~~~~
[\Lambda_{{\varphi}},D'_{\bar{\varphi}}]=-\delta'';\label{phi1}\\
[\Lambda_{{\varphi}},D'']=\delta'_{\bar{\varphi}},~~~~~
[\Lambda_{{\varphi}},D''_{\bar{\varphi}}]=-\delta'.\label{phi}
\end{align}
\end{proposition}

\begin{proof} By Proposition \ref{proba},
$$
[\Lambda_{{\varphi}},D']=\frac{1}{2}([\Lambda_J,D']-
i[\Lambda_K,D'])=\frac{1}{2}(\delta''_J+i\delta''_K)
=\delta''_{\bar{\varphi}},$$
 and
$$\begin{array}{rcl}[\Lambda_{{\varphi}},D'_{\bar{\varphi}}]
&=&\frac{1}{2}([\Lambda_J,D'_J-iD'_K]-
i[\Lambda_K,D'_J-iD'_K])\\
 &=&\frac{1}{2}([\Lambda_J,D'_J]+
i(-[\Lambda_J,D'_K]-[\Lambda_K,D'_J])-[\Lambda_K,D'_K])\\
&=&-\delta''.
\end{array}$$
The rest equations are proved in the same way.
\end{proof}
Let $\Theta_{\bar{\varphi}}$ be the curvature component corresponding to
the operator $D''D'_{\bar{\varphi}}+D'_{\bar{\varphi}}D''.$
Let
$\triangle'_{\bar{\varphi}}=D'_{\bar{\varphi}}\delta'_{\bar{\varphi}}
+\delta'_{\bar{\varphi}}D'_{\bar{\varphi}}$ be another twisted Laplacian operator.
By using Proposition \ref{pro8},
it is easy to prove the following Bochner-Kodaira-Nakano type identity
\begin{proposition}\label{pro10}
\begin{equation}
\triangle''-\triangle'_{\bar{\varphi}}=
[e(\Theta_{\bar{\varphi}}),\Lambda_{\varphi}].
\end{equation}
\end{proposition}

\section{Local expressions of Bochner-Kodaira-Nakano identities}

\setcounter{section}{4}
\setcounter{equation}{0}

Let $(M,I,J,K)$ be a compact hypercomplex manifold  and suppose that $(M,I)$ is a K\"ahler
 manifold with K\"ahler metric $g.$
Given a holomorphic vector bundle $E$  of rank $r$  over $M$ with
Hermitian metric $h,$
let $D=D'+D''$ be the Chern connection of $E$ with  $D''=\bar{\partial}.$  Write
$$D'=\partial+\vartheta,$$ where $\vartheta\in \Omega^{1,0}(End(E))$ is
the connection matrix. By the compatible condition of the connection $D$ and the metric $h,$ we have
$$dh=h\vartheta+\bar{\vartheta}^t h,$$
by comparing the type we get
$$\partial h=h\vartheta,\hskip 1cm \bar{\partial}h=\bar{\vartheta}^t h,$$
it follows that $$\vartheta=h^{-1}\partial h.$$ From
$\partial^2=\bar{\partial}^2=0$ we know
$$\partial \vartheta=-\vartheta\wedge \vartheta, $$
and the Chern curvature $\Theta$ is given by
$$\Theta=\bar{\partial}\vartheta=d\vartheta+\vartheta\wedge \vartheta.$$
 Let $\{z^j\}$ be the local holomorphic
coordinate of $M$ such that the holomorphic coframes in (\ref{holc})
 are represented by $\theta^j=dz^j.$  Let $\{s_{\alpha}\}$ be a
holomorphic  frame, $\{s^{\alpha}\}$ the dual
frame of $E.$ Let $(g_{j\bar{k}})$ and $(h_{\alpha\bar{\beta}})$ be
the hermitian metrics on $M$ and on $E$ respectively, and their
inverses denoted respectively by $(g^{i\bar{j}})$ and
$(h^{\alpha\bar{\beta}}).$ Then the connection and curvature could be expressed by
\begin{equation}\vartheta=\sum_{\alpha,\beta}\vartheta^{\alpha}_{\beta}s_{\alpha}\otimes
s^{\beta},\quad\quad
\Theta=\sum_{\alpha,\beta}\Theta^{\alpha}_{\beta}s_{\alpha}\otimes
s^{\beta}=\sum_{\alpha,\beta,j,k}R^{\alpha}_{\beta j\bar{k}}
dz^j\wedge d\bar{z}^k\otimes s_{\alpha}\otimes
s^{\beta}\label{cc1}\end{equation}
 with
\begin{equation}
\vartheta^{\alpha}_{\beta}=\sum_{\gamma,j}h^{\alpha\bar{\gamma}}
\frac{\partial h_{\beta\bar{\gamma}}}{\partial z^j}dz^j; \quad\quad
\Theta^{\alpha}_{\beta}=\sum_{j,k}R^{\alpha}_{\beta j\bar{k}}
dz^j\wedge
d\bar{z}^k=\sum_{\gamma,j,k}\frac{\partial}{\partial\bar{z}^k}
\Big(h^{\alpha\bar{\gamma}}\frac{h_{\beta\bar{\gamma}}}{\partial
z^j}\Big)d\bar{z}^k\wedge dz^j,\label{cc2}\end{equation} where
\begin{equation} R^{\alpha}_{\beta j\bar{k}}=-\sum_{\gamma}
h^{\alpha\bar{\gamma}}\partial_j\bar{\partial}_k
h_{\beta\bar{\gamma}}+\sum_{\gamma,\lambda,\mu}
h^{\alpha\bar{\gamma}} h^{\lambda\bar{\mu}}\partial_j
h_{\beta\bar{\mu}}\bar{\partial}_k
h_{\lambda\bar{\gamma}}.\label{cc3}\end{equation}

\begin{proposition}\label{cur} \begin{equation}\label{con0}\Theta_J=\bar{\partial}
\vartheta_J;\quad\quad \Theta_K=i\Theta_J;\quad\quad
\Theta_{\bar{\varphi}}=\Theta_J.
\end{equation}
\end{proposition}
\begin{proof}
For any $\xi=\sum_{\alpha}\xi^{\alpha}\otimes s_{\alpha}\in
\Omega^{p,q}(E),$
$$\begin{array}{rcl}
D'_J \xi &=&\sum_{\alpha} (\partial_J \xi^{\alpha})s_{\alpha}+ \sum_{\alpha,\beta}(J^{-1}\vartheta_{\alpha}^{\beta}J\xi^{\alpha})s_{\beta}\\
         &=& (\partial_J+J^{-1}\vartheta J)\xi\\
         &=& (\partial_J+\vartheta_J)\xi.\\
\end{array}$$
Therefore,
\begin{equation}\label{con1}
D'_J D''\xi =\sum_{\alpha} (\partial_J
\bar{\partial}\xi^{\alpha})s_{\alpha}+\sum_{\alpha,\beta}
(J^{-1}\vartheta_{\alpha}^{\beta}J\bar{\partial}\xi^{\alpha})s_{\beta}
\end{equation}
\begin{equation}\label{con2}
D''D'_J \xi
         = \sum_{\alpha}\bar{\partial}(\partial_J\xi^{\alpha})s_{\alpha}+\sum_{\alpha,\beta} \bar{\partial}
         (J^{-1}\vartheta_{\alpha}^{\beta}J)\xi^{\alpha}s_{\beta}-\sum_{\alpha,\beta}(J^{-1}\vartheta_{\alpha}^{\beta}J\bar{\partial}\xi^{\alpha})s_{\beta}
\end{equation}
Since $DJ=D'J=D''J=0$ we have $\bar{\partial}J=J\bar{\partial},$
hence
\begin{equation}\label{con3}\bar{\partial}\partial_J
=\bar{\partial}(J^{-1}\partial J)=J^{-1}\bar{\partial}\partial J,\end{equation}
and
\begin{equation}\label{con4}\partial_J\bar{\partial}
=J^{-1}{\partial}J\bar{\partial}=J^{-1}\partial\bar{\partial} J.\end{equation}
Thus \begin{equation}\label{con5} \bar{\partial}\partial_J +
\partial_J\bar{\partial}=0.\end{equation}
From (\ref{con1}),(\ref{con2}) and (\ref{con5}) we conclude the
first equation of (\ref{con0}).
From (\ref{l}) and (\ref{k}), clearly $\Theta_K=\bar{\partial}
\vartheta_K=i\Theta_J.$ Since by
definition $$\Theta_{\bar{\varphi}}=\frac{1}{2}( \bar{\partial}
\vartheta_J-i\bar{\partial}
\vartheta_K)=\frac{1}{2}(\Theta_J-i\Theta_K),$$
 we have $\Theta_{\bar{\varphi}}=\Theta_J.$
\end{proof}

From (\ref{l}),(\ref{k}) and (\ref{cc2}) we have the following local
expressions of connection and curvature
\begin{equation}\label{cn}(\vartheta_J)^{\alpha}_{\beta}=\sum_{j=1}^n h^{\alpha\bar{\gamma}} \frac{\partial
h_{\beta\bar{\gamma}}}{\partial z^j}d\bar{z}^{j+n}-\sum_{j=1}^n
h^{\alpha\bar{\gamma}} \frac{\partial
h_{\beta\bar{\gamma}}}{\partial z^{j+n}}d\bar{z}^j,\end{equation}
\begin{equation}\label{curr1}(\Theta_J)^{\alpha}_{\beta}=\sum_{k=1}^{2n}\sum_{j=1}^n\frac{\partial}{{\partial}\bar{z}^k}
\Big(h^{\alpha\bar{\gamma}} \frac{\partial
h_{\beta\bar{\gamma}}}{\partial z^j}\Big)d\bar{z}^k\wedge
d\bar{z}^{j+n}-\sum_{k=1}^{2n}\sum_{j=1}^n
\frac{\partial}{{\partial}\bar{z}^k} \Big(h^{\alpha\bar{\gamma}}
\frac{\partial h_{\beta\bar{\gamma}}}{\partial
z^{j+n}}\Big)d\bar{z}^k\wedge d\bar{z}^j.\end{equation}
Using (\ref{cc3}), we could write the curvature components of $\Theta_J$ simply as
\begin{equation}\label{curr2}(\Theta_J)^{\alpha}_{\beta}=-\sum_{k=1}^{2n}\sum_{j=1}^n R^{\alpha}_{\beta j\bar{k}} d\bar{z}^k\wedge
d\bar{z}^{j+n}+\sum_{k=1}^{2n}\sum_{j=1}^n R^{\alpha}_{\beta
j+n\bar{k}}d\bar{z}^k\wedge d\bar{z}^j.\end{equation}
Therefore,\begin{equation}\label{curr3}\Theta_J=\sum_{k=1}^{2n}\sum_{j=1}^n\Big( -R^{\alpha}_{\beta j\bar{k}} d\bar{z}^k\wedge
d\bar{z}^{j+n}\otimes
s_{\alpha}\otimes s^{\beta}+ R^{\alpha}_{\beta
j+n\bar{k}}d\bar{z}^k\wedge d\bar{z}^j\otimes
s_{\alpha}\otimes s^{\beta}\Big).\end{equation}

 The proof of the following lemma is simply via using the commuting relation (\ref{kd0})
(\ref{kd2}) and (\ref{kd3}), we omit it here for brevity.
\begin{lemma}\label{id1} Let the operators $e_k,\bar{e}_k$ be the wedge operators defined as in
(\ref{www1}),(\ref{www2}) and $i_k,\bar{i}_k$ their adjoint operators, then for any integer
$1\leq p,q,k\leq 2n,$
\begin{equation} [\bar{e}_p \bar{e}_q, i_{k}i_{k+n}]=0;\label{comij1}\end{equation}
\begin{eqnarray}[e_p \bar{e}_q, i_{k}\bar{i}_{k}]=\left\{\begin{array}{rcl}0,& p,q\not=k;\\
-2\bar{e}_q\bar{i}_{k},& p=k,q\not=k;\\
-2{e}_p{i}_{k},& q=k, p\not=k;\\
4-2{e}_k{i}_k-2\bar{e}_{k}\bar{i}_{k}, &p=q=k;
\end{array}\right.\label{comij2}\end{eqnarray}
\begin{eqnarray}[\bar{e}_p \bar{e}_q,
\bar{i}_{k}\bar{i}_{k+n}]=\left\{\begin{array}{rcl}0,& p,q\not=k,
k+n;\\
-2\bar{e}_q\bar{i}_{k+n},& p=k,q\not=k, k+n;\\
2\bar{e}_p\bar{i}_{k+n},& q=k, p\not=k, k+n;\\
2\bar{e}_q\bar{i}_k,& p=k+n,q\not=k, k+n;\\
-2\bar{e}_p\bar{i}_k,&q=k+n, p\not=k, k+n;\\
4-2\bar{e}_k\bar{i}_k-2\bar{e}_{k+n}\bar{i}_{k+n}, &p=k,q=k+n;\\
-4+2\bar{e}_k\bar{i}_k+2\bar{e}_{k+n}\bar{i}_{k+n},&q=k, p= k+n.\\
\end{array}\right.\label{comij3}\end{eqnarray}
\end{lemma}

For any $E$-valued differential form $\xi\in\Omega^{p,q}(E),$ write
$$\xi=\sum_{\alpha}\xi^{\alpha}\otimes
s_{\alpha}=\sum_{P,Q,\alpha}\xi_{P,Q,\alpha}\theta^P\wedge\bar{\theta}^Q\otimes
s_{\alpha},$$ where the length $|P|=p$ and $|Q|=q,$ and
$$\xi^{\alpha}=\sum_{P,Q}\xi_{P,Q,\alpha}\theta^P\wedge\bar{\theta}^Q\in\Omega^{p,q}(M).$$

\begin{proposition}\label{id2}
\begin{equation} ({e}_q{i}_{k} \xi,\xi)_E =2\sum_{\alpha,P,Q}\sum_{(Sk)=(Tq)=P}
\xi_{Sk,Q,\alpha}\bar{\xi}_{Tq,Q,\alpha}.\label{innn}\end{equation}
\end{proposition}
\begin{proof} Note that
 $i_k
( \theta^P\wedge\bar{\theta}^Q)=0$  and $i_k (\theta^k\wedge
\theta^P\wedge\bar{\theta}^Q)=2 \theta^P\wedge\bar{\theta}^Q$ if
$k\notin P.$
 Since $({e}_q{i}_{k} u,u)_E =({i}_{k} \xi,i_q \xi)_E $ and
 $(\theta^k\wedge \theta^P\wedge\bar{\theta}^Q,\theta^k\wedge
 \theta^P\wedge\bar{\theta}^Q)_E=2(\theta^P\wedge\bar{\theta}^Q,
 \theta^P\wedge\bar{\theta}^Q)_E,$ Proposition \ref{id2} follows.
\end{proof}

 The formula in the following proposition appeared in section 4 of \cite{de} without proof. (Note the expression is
 a little different since it uses unit orthogonal frame in \cite{de} while here
 $\langle\theta^k,\theta^k\rangle=2$ ). We will
 give a very simple proof here.
\begin{proposition}\label{demm}
\begin{equation}\begin{array}{rcl}\frac{1}{2}\lge [e(i\Theta),\Lambda]\xi,\xi\rg
&=&\sum_{\alpha,\beta,P,Q}\sum_{(S\bar{k})=(T\bar{q})=Q}R^{\alpha}_{\beta
k\bar{q}}\xi_{P,S\bar{k},\alpha}\bar{\xi}_{P,T\bar{q},\beta}\\
&&+\sum_{\alpha,\beta,P,Q}\sum_{(S{k})=(T{q})=P}R^{\alpha}_{\beta
p\bar{k}}\xi_{Sk,Q,\alpha}\bar{\xi}_{Tq,Q,\beta}\\
&&-\sum_{\alpha,\beta,k,P,Q}R^{\alpha}_{\beta
k\bar{k}}\xi_{P,Q,\alpha}\bar{\xi}_{P,Q,\beta}\\
\end{array}\label{iii}\end{equation}
\end{proposition}
\begin{proof}
By (\ref{cc1}) and (\ref{cc2}),
$$e(i\Theta)=i\sum_{\alpha,\beta,p,q}R^{\alpha}_{\beta p\bar{q}}
e_p \bar{e}_q\otimes s_{\alpha}\otimes s^{\beta},$$ recall that
$\Lambda=\frac{i}{2}\sum _{k}i_{k}\bar{i}_{k},$ hence
$$[e(i\Theta),\Lambda]\xi =-\frac{1}{2}\sum_{p,q,\alpha,\beta,k,P,Q}
R^{\alpha}_{\beta p\bar{q}} ([e_p \bar{e}_q,
i_{k}\bar{i}_{k}]\xi^{\beta}) s_{\alpha}.$$ By (\ref{comij2}) of Proposition
\ref{id1},
$$\begin{array}{rcl}[e(i\Theta),\Lambda]\xi &=&\sum_{\alpha,\beta,k,P,Q}\sum_{q\not=k}
R^{\alpha}_{\beta k\bar{q}} (\bar{e}_q\bar{i}_{k}\xi^{\beta})
s_{\alpha}+\sum_{\alpha,\beta,k,P,Q}\sum_{p\not=k} R^{\alpha}_{\beta
p\bar{k}} (e_p i_{k}\xi^{\beta}) s_{\alpha}\\
&&+ \sum_{\alpha,\beta,k,P,Q} R^{\alpha}_{\beta k\bar{k}}
(\bar{e}_k\bar{i}_{k}\xi^{\beta})
s_{\alpha}+\sum_{\alpha,\beta,k,P,Q}R^{\alpha}_{\beta
k\bar{k}} (e_k i_{k}\xi^{\beta}) s_{\alpha}\\
&&-2\sum_{\alpha,\beta,k,P,Q}R^{\alpha}_{\beta
k\bar{k}} \xi^{\beta} s_{\alpha}\\
 &=&\sum_{\alpha,\beta,q,k,P,Q}
R^{\alpha}_{\beta k\bar{q}} (\bar{e}_q\bar{i}_{k}\xi^{\beta})
s_{\alpha}+\sum_{\alpha,\beta,p,k,P,Q} R^{\alpha}_{\beta
p\bar{k}} (e_p i_{k}\xi^{\beta}) s_{\alpha}\\
&&-2\sum_{\alpha,\beta,k,P,Q}R^{\alpha}_{\beta
k\bar{k}} \xi^{\beta} s_{\alpha}.\\
\end{array}$$
 Taking inner products of both sides of the above equation with $\xi,$ and
using (\ref{innn}), we get immediately the equation (\ref{iii}).

\end{proof}
\begin{proposition}\label{naj}
\begin{equation}\begin{array}{rcl}
\frac{1}{2}\lge[e(\Theta_J),\Lambda_J]\xi,\xi\rg
&=&\sum_{\alpha,\beta,P,Q}\sum_{p,q=1}^{2n}\sum_{(S\bar{q})=(T\bar{p})=Q}R^{\alpha}_{\beta
p\bar{q}}\xi_{P,S\bar{q},\alpha}\bar{\xi}_{P,T\bar{p},\beta}\\
&&-\sum_{\alpha,\beta,P,Q}\sum_{k=1}^{2n}R^{\alpha}_{\beta
k\bar{k}}\xi_{P,Q,\alpha}\bar{\xi}_{P,Q,\beta}\\
&&+\sum_{\alpha,\beta,P,Q}\sum_{p,q=1}^{n}\Big( \sum_{(S\bar{q})=(T\bar{p})=Q}R^{\alpha}_{\beta
p+n\overline{q+n}}\xi_{P,S\bar{q},\alpha}\bar{\xi}_{P,T\bar{p},\beta}\\
&&+\sum_{(S\overline{p+n})=(T\overline{q+n})=Q}R^{\alpha}_{\beta
p\bar{q}}\xi_{P,S\overline{p+n},\alpha}\bar{\xi}_{P,T\overline{q+n},\beta}\Big) \\
&&+\sum_{\alpha,\beta,P,Q}\sum_{p,q=1}^n
\Big(\sum_{(S\bar{p})=(T\overline{q+n})=Q}R^{\alpha}_{\beta
p\overline{q+n}}\xi_{P,S\bar{p},\alpha}\bar{\xi}_{P,T\overline{q+n},\beta}\\
&&+\sum_{(S\overline{p+n})=(T\bar{q})=Q}R^{\alpha}_{\beta
p+n\bar{q}}\xi_{P,S\overline{p+n},\alpha}\bar{\xi}_{P,T\bar{q},\beta}\\
&&-\sum_{(S\bar{q})=(T\overline{p+n})=Q}R^{\alpha}_{\beta
p\overline{q+n}}\xi_{P,S\bar{q},\alpha}\bar{\xi}_{P,T\overline{p+n},\beta}\\
&&-\sum_{(S\overline{q+n})=(T\bar{p})=Q}R^{\alpha}_{\beta
p+n\bar{q}}\xi_{P,S\overline{q+n},\alpha}\bar{\xi}_{P,T\bar{p},\beta}
\Big).\\
\end{array}\label{jjj}\end{equation}
\end{proposition}

\begin{proof} Using the expression (\ref{curr3}) of the curvature $\Theta_J,$ we have
$$e(\Theta_J)=\sum_{\alpha,\beta=1}^r\sum_{p,q=1}^{n} (-R^{\alpha}_{\beta p\bar{q}}
\bar{e}_q \bar{e}_{p+n}
-R^{\alpha}_{\beta p\bar{q+n}} \bar{e}_{q+n}
\bar{e}_{p+n}+R^{\alpha}_{\beta
p+n\bar{q}}\bar{e}_q\bar{e}_{p}+R^{\alpha}_{\beta
p+n\bar{q+n}}\bar{e}_{q+n}\bar{e}_{p})\otimes s_{\alpha}\otimes
s^{\beta},$$ recall that by (\ref{lamj}),
$$\Lambda_J=\frac{1}{2}\sum _{k=1}^n
(i_{k}i_{k+n}+\bar{i}_{k}\bar{i}_{k+n}).$$ By (\ref{comij1}) of Proposition
\ref{id1}, for any integers $1\leq p,q,k\leq 2n,$ we have
\begin{equation}[\bar{e}_q\bar{e}_{p},i_{k}i_{k+n}]=0, \label{bsic}\end{equation}
 Therefore
\begin{equation}\begin{array}{rcl}[e(\Theta_J),\Lambda_J]\xi
&=&\frac{1}{2}\sum_{\alpha,\beta=1}^r\sum_{p,q,k=1}^{n}
\sum_{P,Q}\Big( -R^{\alpha}_{\beta p\bar{q}} [\bar{e}_q
\bar{e}_{p+n},
\bar{i}_{k}\bar{i}_{k+n}]\xi^{\beta}\\
&&-R^{\alpha}_{\beta p\bar{q+n}}
[\bar{e}_{q+n} \bar{e}_{p+n},\bar{i}_{k}\bar{i}_{k+n}]\xi^{\beta}\\
&&+R^{\alpha}_{\beta
p+n\bar{q}}[\bar{e}_q\bar{e}_{p},\bar{i}_{k}\bar{i}_{k+n}]\xi^{\beta}+R^{\alpha}_{\beta
p+n\bar{q+n}}[\bar{e}_{q+n}\bar{e}_{p},\bar{i}_{k}\bar{i}_{k+n}]\xi^{\beta}\Big)s_{\alpha}.\\
\end{array}\label{s}\end{equation}
By (\ref{comij3}) of Proposition \ref{id1}, if neither $p$ nor $q$ takes values $k,k+n,$ we have
$$[\bar{e}_p\bar{e}_{q},\bar{i}_{k}\bar{i}_{k+n}]=
[\bar{e}_p\bar{e}_{q},\bar{i}_{k}\bar{i}_{k+n}]=0,$$ hence
\begin{equation}[e(\Theta_J),\Lambda_J]\xi=\sum_{p=k;q\not=k}+\sum_{p\not=k;q=k}
+\sum_{p=k;q=k}.
\label{su}\end{equation}
Take $\sum_{p=k;q\not=k}$ for an example, here it means in the summation of equation (\ref{s}), we add only a restrict
condition $p=k, q\not=k.$ In others words,   $\sum_{p=k;q\not=k}$  is
a summation whose terms are the same as in equation (\ref{s}), where
 indices $\alpha,\beta,P,Q, p,q,k$ vary  the same range as
 in equation (\ref{s}) except that $p=k, q\not=k.$

By (\ref{comij3}) of Proposition \ref{id1}, for $j\not=k,k+n,$ we have
$$[\bar{e}_j\bar{e}_{k},\bar{i}_{k}\bar{i}_{k+n}]=2\bar{e}_j \bar{i}_{k+n},\hskip 1cm
[\bar{e}_j\bar{e}_{k+n},\bar{i}_{k}\bar{i}_{k+n}]=-2\bar{e}_j \bar{i}_{k},$$ hence
\begin{equation}\begin{array}{rcl}
\sum_{p=k;q\not=k}&=&  \sum_{\alpha,\beta,P,Q,q\not=k}\Big(
R^{\alpha}_{\beta k\bar{q}} \bar{e}_q
\bar{i}_{k}\xi^{\beta}+R^{\alpha}_{\beta k\overline {q+n}}
\bar{e}_{q+n} \bar{i}_{k}\xi^{\beta}\\
&&+R^{\alpha}_{\beta
k+n\bar{q}}\bar{e}_q\bar{i}_{k+n}\xi^{\beta}+R^{\alpha}_{\beta
k+n\overline{q+n}}\bar{e}_{q+n}\bar{i}_{k+n}\xi^{\beta}\Big)s_{\alpha}.\\
\end{array}\label{j1}\end{equation}
By the same reasons,
\begin{equation}\begin{array}{rcl}
\sum_{p\not=k;q=k}&=& \sum_{\alpha,\beta,P,Q,p\not=k}\Big(
R^{\alpha}_{\beta p\bar{k}} \bar{e}_{p+n}
\bar{i}_{k+n}\xi^{\beta}-R^{\alpha}_{\beta p\overline {k+n}}
\bar{e}_{p+n} \bar{i}_{k}\xi^{\beta}\\
&&-R^{\alpha}_{\beta
p+n\bar{k}}\bar{e}_p\bar{i}_{k+n}\xi^{\beta}+R^{\alpha}_{\beta
p+n\overline{k+n}}\bar{e}_{p}\bar{i}_{k}\xi^{\beta}\Big)s_{\alpha}.\\
\end{array}\label{j2}\end{equation}
Since by (\ref{comij3}) of Proposition \ref{id1},
$[\bar{e}_k\bar{e}_{k+n},\bar{i}_{k}\bar{i}_{k+n}]=4-2\bar{e}_k
\bar{i}_{k}-2\bar{e}_{k+n} \bar{i}_{k+n},$ we have
\begin{equation}
\sum_{p=k;q=k}=  -\sum_{\alpha,\beta,P,Q}(R^{\alpha}_{\beta
k\bar{k}}+R^{\alpha}_{\beta k+n\overline {k+n}})(2- \bar{e}_{k}
\bar{i}_{k}- \bar{e}_{k+n}
\bar{i}_{k+n})\xi^{\beta}s_{\alpha}.\\\label{j3}
\end{equation}
Note the equation (\ref{j1}) could be rewritten as
\begin{equation}\begin{array}{rcl}
\sum_{p=k;q\not=k}&=&  \sum_{\alpha,\beta,P,Q,p,q}\Big(
R^{\alpha}_{\beta p\bar{q}} \bar{e}_q
\bar{i}_{p}\xi^{\beta}+R^{\alpha}_{\beta p\overline {q+n}}
\bar{e}_{q+n} \bar{i}_{p}\xi^{\beta}\\
&&+R^{\alpha}_{\beta
p+n\bar{q}}\bar{e}_q\bar{i}_{p+n}\xi^{\beta}+R^{\alpha}_{\beta
p+n\overline{q+n}}\bar{e}_{q+n}\bar{i}_{p+n}\xi^{\beta}\Big)s_{\alpha}\\
&&- \sum_{\alpha,\beta,P,Q,k}\Big( R^{\alpha}_{\beta k\bar{k}}
\bar{e}_k \bar{i}_{k}\xi^{\beta}+R^{\alpha}_{\beta k\overline {k+n}}
\bar{e}_{k+n} \bar{i}_{k}\xi^{\beta}\\
&&+R^{\alpha}_{\beta
k+n\bar{k}}\bar{e}_k\bar{i}_{k+n}\xi^{\beta}+R^{\alpha}_{\beta
k+n\overline{k+n}}\bar{e}_{k+n}\bar{i}_{k+n}\xi^{\beta}\Big)s_{\alpha},\\
\end{array}\label{j4}\end{equation}
we remark where in the first summation of right hand side of equation (\ref{j4}), we have change the index $k$ to $q$
(note $p,q,k$ have equal positions since all  of them vary form $1$ to $n$).
Similarly, the equation (\ref{j2}) could be rewritten as
\begin{equation}\begin{array}{rcl}
\sum_{p\not=k;q=k}&=& \sum_{\alpha,\beta,P,Q,p,q}\Big(
R^{\alpha}_{\beta p\bar{q}} \bar{e}_{p+n}
\bar{i}_{q+n}\xi^{\beta}-R^{\alpha}_{\beta p\overline {q+n}}
\bar{e}_{p+n} \bar{i}_{q}\xi^{\beta}\\
&&-R^{\alpha}_{\beta
p+n\bar{q}}\bar{e}_p\bar{i}_{q+n}\xi^{\beta}+R^{\alpha}_{\beta
p+n\overline{q+n}}\bar{e}_{p}\bar{i}_{q}\xi^{\beta}\Big)s_{\alpha}\\
&&-\sum_{\alpha,\beta,P,Q,k}\Big( R^{\alpha}_{\beta k\bar{k}}
\bar{e}_{k+n} \bar{i}_{k+n}\xi^{\beta}-R^{\alpha}_{\beta k\overline
{k+n}}
\bar{e}_{k+n} \bar{i}_{k}\xi^{\beta}\\
&&-R^{\alpha}_{\beta
k+n\bar{k}}\bar{e}_k\bar{i}_{k+n}\xi^{\beta}+R^{\alpha}_{\beta
k+n\overline{k+n}}\bar{e}_{k}\bar{i}_{k}\xi^{\beta}\Big)s_{\alpha}.\\
\end{array}\label{j5}\end{equation}
Combining Eq. (\ref{j5}),  Eq. (\ref{j4}) and Eq. (\ref{j3}), we get
from Eq. (\ref{su}) that
\begin{equation}\begin{array}{rcl}[e(\Theta_J),\Lambda_J]\xi
&=&\sum_{\alpha,\beta=1}^r\sum_{p,q=1}^{n}
\sum_{P,Q}\Big(R^{\alpha}_{\beta p\bar{q}} \bar{e}_q
\bar{i}_{p}\xi^{\beta}+R^{\alpha}_{\beta p\overline {q+n}}
\bar{e}_{q+n} \bar{i}_{p}\xi^{\beta}\\
&&+R^{\alpha}_{\beta
p+n\bar{q}}\bar{e}_q\bar{i}_{p+n}\xi^{\beta}+R^{\alpha}_{\beta
p+n\overline{q+n}}\bar{e}_{q+n}\bar{i}_{p+n}\xi^{\beta}\Big)s_{\alpha}\\
&&+R^{\alpha}_{\beta p\bar{q}} \bar{e}_{p+n}
\bar{i}_{q+n}\xi^{\beta}-R^{\alpha}_{\beta p\overline {q+n}}
\bar{e}_{p+n} \bar{i}_{q}\xi^{\beta}\\
&&-R^{\alpha}_{\beta
p+n\bar{q}}\bar{e}_p\bar{i}_{q+n}\xi^{\beta}+R^{\alpha}_{\beta
p+n\overline{q+n}}\bar{e}_{p}\bar{i}_{q}\xi^{\beta}\Big)s_{\alpha}\\
&&-2\sum_{\alpha,\beta=1}^r\sum_{k=1}^{2n}\sum_{P,Q}R^{\alpha}_{\beta
k\bar{k}}\xi^{\beta}s_{\alpha}.
\end{array}\label{sss}\end{equation}
If we rearrange the summation range of indices $p,q$, we get
\begin{equation}\begin{array}{rcl}[e(\Theta_J),\Lambda_J]\xi
&=&\sum_{\alpha,\beta=1}^r\sum_{p,q=1}^{2n}
\sum_{P,Q} R^{\alpha}_{\beta p\bar{q}} \bar{e}_q
\bar{i}_{p}\xi^{\beta}\\
&&-2\sum_{\alpha,\beta=1}^r\sum_{k=1}^{2n}\sum_{P,Q}R^{\alpha}_{\beta
k\bar{k}}\xi^{\beta}s_{\alpha}\\
&&+\sum_{\alpha,\beta=1}^r\sum_{p,q=1}^{n}
\sum_{P,Q} \Big(R^{\alpha}_{\beta p\bar{q}} \bar{e}_{p+n}
\bar{i}_{q+n}\xi^{\beta}
+R^{\alpha}_{\beta
p+n\overline{q+n}}\bar{e}_{p}\bar{i}_{q}\xi^{\beta}\Big)s_{\alpha}\\
&&+ \sum_{\alpha,\beta=1}^r\sum_{p,q=1}^{n}
\sum_{P,Q}\Big(R^{\alpha}_{\beta p\overline {q+n}}
\bar{e}_{q+n} \bar{i}_{p}\xi^{\beta}+R^{\alpha}_{\beta
p+n\bar{q}}\bar{e}_q\bar{i}_{p+n}\xi^{\beta}\\
&&-R^{\alpha}_{\beta p\overline {q+n}}
\bar{e}_{p+n} \bar{i}_{q}\xi^{\beta}-R^{\alpha}_{\beta
p+n\bar{q}}\bar{e}_p\bar{i}_{q+n}\xi^{\beta}\Big)s_{\alpha}\\
\end{array}\label{sm}\end{equation}
Taking inner products of both sides of equation (\ref{sm}) with $\xi,$ and
using (\ref{innn}), we get immediately the equation (\ref{jjj}).
\end{proof}

From the definitions of $\Lambda_J,\Lambda_K$ and
$\Lambda_{\varphi}$ we have
\begin{equation}\Lambda_{\varphi}=\sum_{k=1}^n\bar{i}_k\bar{i}_{k+n}.
\end{equation}
Recall that by Proposition  \ref{cur}, $\Theta_K=i\Theta_J,\Theta_{\bar{\varphi}}=\Theta_J.$
By the proof of Proposition \ref{naj}, in particular, from the equation (\ref{bsic}), we have
\begin{equation}[e(\Theta_K),\Lambda_K]=[e(i\Theta_J),\frac{i}{2}
\sum_{k=1}^n\bar{i}_k\bar{i}_{k+n}]=-[e(\Theta_J),\Lambda_J],\label{eqc1}
\end{equation}
and similarly
\begin{equation}[e(\Theta_{\bar{\varphi}}),\Lambda_{\varphi}]=[e(\Theta_J),
\sum_{k=1}^n\bar{i}_k\bar{i}_{k+n}]=2[e(\Theta_J),\Lambda_J].\label{eqc2}
\end{equation}
By Proposition \ref{pro7} and \ref{pro10},  Eq. (\ref{eqc1}) and Eq.
(\ref{eqc2}) we get
\begin{proposition}\label{nouse}
$$\begin{array}{rcl}
\triangle''-\triangle'_{J}&=&
[e(\Theta_J),\Lambda_J];\\
\triangle''-\triangle'_{K}&=&
[e(\Theta_J),\Lambda_J];\\
\triangle''-\triangle'_{\bar{\varphi}}&=&
2[e(\Theta_J),\Lambda_J].
\end{array}$$
\end{proposition}

\section{Vanishing theorems for hypercomplex K\"ahler manifolds}
\setcounter{section}{5}
\setcounter{equation}{0}

Let $(M,I,J,K)$ be a compact hypercomplex K\"ahler manifold with K\"ahler metric $h$ and K\"ahler form $\omega=\omega_I.$
 The ideal to establish vanishing theorems  for $E$-valued
 Dolbeault cohomology groups via
 using the Bochner-Kodaira-Nakano identities is very simple. We use the first formula of Proposition \ref{nouse}
\begin{equation}\triangle''-\triangle'_{J}=
[e(\Theta_J),\Lambda_J]\label{bk1}\end{equation} as an example.
If $\xi$ is an arbitrary
$E$-valued $(p,q)$-form, then  an integration by part of the formula
$$ ((\triangle''-\triangle'_{J})\xi, \xi)_E=
([e(\Theta_J),\Lambda_J]\xi, \xi)_E$$
and noting that $(\triangle'_{J}\xi, \xi)_E=||D'_J\xi||^2_E +||\delta'_J\xi||^2_E\geq 0$
 yields
\begin{equation} ||D''\xi||_E +||D''^*\xi||_E\geq \int_M \langle [\Theta_J
,\Lambda_J]\xi,\xi\rangle.\label{bknnn}\end{equation} If $\xi\in
H^{p,q}(E),$ then $\xi$ is $\Delta''$-harmonic and hence
$D''\xi=D''^*\xi=0$ by the Hodge theory.  Furthermore if  that
$[\Theta_J,\Lambda_J]$ is positive definite everywhere on
$\Omega^{p,q}(E),$ then $\xi=0.$ Hence $H^{p,q}(M,E)\cong
H^q(M,\Omega^p (E))\cong H^{p,q}(E)=0.$ Thus we get a vanishing
cohomology group.
 Therefore,  using (\ref{bk1}) to prove a vanishing theorem for $E$-valued
 Dolbeault cohomology groups, the key point is
 to find conditions under which the operator
 $[\Theta_J,\Lambda_J]$ is positive definite.

We can see from above reasoning, the second and third formulae of  Proposition \ref{nouse} produce no new
 vanishing theorems since their left hand sides are the same up to a positive constant.
For a hypercomplex K\"ahler manifold $(M,I,J,K),$ the three complex structures $I,J,K$ have symmetric positions, however
they are not independent of each other and related by $IJ=-JI=K.$ This may account that only
two Bochner-Kodaira-Nakano identities,
the formula (\ref{bkn}) and one formula of Proposition \ref{nouse}
 produce different vanishing theorems. Note however,
the computations of the proof (though we don't give its proof) of the last formula of  Proposition
\ref{nouse} is simpler than the other two equations.

By (\ref{cc1}), the Chern curvature form of $E$ is given by
$$\Theta=\Theta(E)=\sum_{\alpha,\beta}\Theta^{\alpha}_{\beta}s_{\alpha}\otimes
s^{\beta}=\sum_{\alpha,\beta,j,k}R^{\alpha}_{\beta j\bar{k}}
dz^j\wedge d\bar{z}^k\otimes s_{\alpha}\otimes s^{\beta}.$$ The
first Chern class $c_1(E)\in H^2(M, \R)$ is a cohomology class which
has a representation via using the
 Chern curvature form
\begin{equation}c_1(E)=\frac{1}{2\pi}{\rm Tr}_E \big(i\Theta
(E)\big)= \frac{i}{2\pi} \Theta (\det
(E)).\label{23}\end{equation} Conversely,  any $2$-form representing the first Chern class
$c_1(E)$ is in fact the Chern curvature form of
some hermitian metric  on $\det E$ (up to a constant).
 In local
coordinate we have
\begin{equation}i\Theta (\det (E))= i\sum_{j,k}R_{j\bar{k}}dz^j\wedge d\bar{z}^k
=-\partial\bar{\partial}\log \det(h_{\alpha\bar{\beta}})
\label{24}\end{equation} with $R_{j\bar{k}}=
\sum_{\alpha}R^{\alpha}_{\alpha j\bar{k}}.$
   In particular, if
$E$ is a line bundle  then its curvature form represents its first
Chern class up to a constant $\frac{1}{2\pi}.$

In \cite{ya}, we introduce the following notion for semipositive holomorphic vector bundles. Base on it
and the formula (\ref{bkn}) for K\"ahler manifolds, we get some new vanishing theorems.

\begin{definition}\label{defy} {\rm A holomorphic vector bundle $E$ of rank
$r$ with hermitian metric $h$ on a compact complex manifold $M$ of
complex dimension $n$ is called {\it $(k,s)$-positive} for $1\leq
s\leq r,$ if the following holds for any $x\in M:$ For any $s$-tuple
vectors $ v^j \in V, 1\leq j\leq s,$ where $V=E_x$ ({\rm resp.} $T_x
M$), the hermitian form
$$Q_x (\bullet, \bullet )=i\Theta(E)(\sum^s_{j=1}\cdot\otimes v^j,\sum^s_{j=1}\cdot\otimes v^j),
\quad \bullet\in W^{\oplus s},~~\cdot\in W $$
 defined on $W^{\oplus s}$ is semipositive and the dimension
 of its kernel is at most $k,$ where $W=T_x M$  {\rm (resp.} $E_x${\rm )}.}
 \end{definition}

  Clearly the $(0,s)$-positivity  is equivalent to the Demailly $s$-positivity \cite{de}
  and the
Nakano positivity \cite{ss} is equivalent to the $(0, s)$-positivity if $s\geq
{\rm min}\{n,r\}.$
  The $(0, 1)$-positivity is equivalent to the Griffiths
  positivity. For general integer $k,$ the $(k, 1)$ positivity is
  a semipositive version of the Griffiths positivity \cite{gr}.  A
  holomorphic vector bundle $E$ of arbitrary rank is called {\it Griffiths k-positive} if
if it is $(k, 1)$-positive.

\begin{theorem}\label{hirank}  Let $M$ be a compact hypercomplex K\"ahler manifold of
dimension $4n$ and let $E$ be a hermitian holomorphic vector bundle of rank
$r$ on $M$ such that $E$ is $(k,s)$-positive.   Then

{\rm (i)} $H^{p}(M, E)=0$,\quad for \quad $p> k$~~and~~$s\geq {\rm
min}\{2n-p+1,r\};$

{\rm (ii)} If in addition $k\leq 2n-1,$ then for $s\geq {\rm
min}\{2n-p+1,r\}$ and any nonnegative integer $p,$
$$H^{2n}(M,\Omega^p\otimes E)=0.$$
\end{theorem}
 \begin{proof}~ (i) follows from Theorem 3.9 of \cite{ya} together
  with the fact that the anticanonical bundle of $M$
is trivial. It suffices to prove (ii).
 For any $E$-valued  $(p,2n)$-form $\xi=\sum_{P,\alpha}
\xi^{\alpha}_{P,Q}\theta^P\wedge \bar{\theta}^1\wedge\cdots\wedge \bar{\theta}^{2n}\wedge
\otimes e_{\alpha}\in\Omega^{p,2n}(E)$ where $Q=12\cdots 2n$ is a fixed index. By Definition
\ref{hirank}, the hermitian form $i\Theta
(E)$ is semipositive on $\Omega^{p,2n}(E)$ if $s\geq {\rm min}\{2n-p+1,r\},$ we could diagonalize  it in some local orthogonal frames such that
$R_{\alpha\bar{\beta}j\bar{k}}=\lambda^{\alpha}_j\delta_{\alpha\beta}\delta_{jk}.$
Here $(\lambda^{\alpha}_j)_{1\leq j\leq
2n,1\leq\alpha\leq r}$ are non-negative and for a fixed $\alpha,$ without loss generality we assume that
$\lambda^{\alpha}_1\leq
\lambda^{\alpha}_2\leq \cdots\leq \lambda^{\alpha}_{2n}$ with
$\lambda^{\alpha}_{k+1}>0,$ in particular $\lambda^{\alpha}_{2n}>0$ since $k\leq 2n-1.$ Put
$\lambda=\min\{\lambda^{\alpha}_{2n}|1\leq \alpha\leq r\}.$ Then
$\lambda$ is a positive number.
Note in the present situation the first two terms cancel each other in the first summation and the last
 summation vanishes
in the three big summations of the Bochner-Kodaira-Nakano formula (\ref{jjj}). Hence
\begin{equation}\begin{array}{rcl}
\frac{1}{2}\lge[e(\Theta_J),\Lambda_J]\xi,\xi\rg
&=&\sum_{\alpha, P}\sum_{j=1}^{n}\Big( \sum_{j\in Q}\lambda^{\alpha}_j
|\xi^{\alpha}_{P,Q}|^2+\sum_{j+n\in Q}\lambda^{\alpha}_{j+n}
|\xi^{\alpha}_{P,Q}|^2) \\
 &\geq &\sum_{P}\sum_{\alpha} (\lambda^{\alpha}_1 +
\lambda^{\alpha}_2 +\cdots +
\lambda^{\alpha}_{2n})|\xi^{\alpha}_{P,Q}|^2\\
&\geq & \lambda(\sum_{P}\sum_{\alpha}|\xi^{\alpha}_{P,Q}|^2)\\
&=& \lambda |\xi|^2.\\
\end{array}\label{kkk}\end{equation}
 Thus  $[e(\Theta_J),\Lambda_J]$ is positive
definite on $E$-valued $(p,2n)$-forms. So we have
$H^{2n}(M,\Omega^p\otimes E)=0$ for any $p$ and $s\geq {\rm
min}\{2n-p+1,r\}.$
\end{proof}

 A holomorphic
line bundle $B$ on $M$ is called {\it
$k$-positive}, if there is a hermitian metric on $B$
 such that its first Chern class $c_1(B)$ is semi-positive and has at least
 $n-k$ positive eigenvalues \cite{gi},\cite{ss}.
 If $E$ is a  holomorphic
line bundle (denoted it by $B$ for the difference), then in Definition \ref{defy}
only $(k,1)$-positivity is applicable for  $B,$ and clearly
 $B$ is $(k,1)$-positive (or Griffiths $k$-positive) if and only if it is $k$-positive,
since the  first Chern class  has a representation by
its chern curvature form up to a positive constant.

  \begin{theorem} \label{mainr}
 Let $B$ be a $k$-positive holomorphic
 line bundle on a compact hypercomplex K\"ahler manifold $M.$  Then

{\rm (i)} $H^p(M,\Omega^q\otimes B)=0,$\quad for \quad $p+q>2n+k;$

{\rm (ii)} $H^p (M,\Omega^q\otimes B)=0,$ \quad for
$p>n+[\frac{k}{2}]$ and any nonnegative  integer $q.$
\end{theorem}

\begin{proof}
We get (i) by using  the Gigante-Girbau's vanishing theorem on K\"ahler manifolds \cite{gi},
a simple proof is given in Theorem 2.4 of \cite{ya}.
 (i) is proved via using (\ref{bkn}) and changing the K\"ahler
metric on $M.$
(ii) is proved in the same way as (i) by using the Bochner-Kodaira-Nakano identities
in Proposition \ref{nouse}. Here we give a proof of (ii) in the following paragraph.

Choose a holomorphic local coordinate system at each point $x\in M,$ which
diagonalizes simultaneously the hermitian form $\omega(x)$ and
$i\Theta (x)$ since both of them are semipositive, such that
$$\begin{array}{rcl}\omega(x)=i\sum_{j=1}^n \mu_j (x) dz^j\wedge
d\bar{z}^j,&& i\Theta(x)=i\sum_{j=1}^n \nu_j (x) dz^j\wedge
d\bar{z}^j.\end{array}$$
 Without loss of generality, assume that  ${\nu_1}(x)/{\mu_1}(x)
 \leq \cdots \leq {\nu_{2n}}(x)/{\mu_{2n}}(x).$
 Then for any $(p,q)$-form $\xi=\sum_{P,Q}\xi_{P,Q}\theta^P\wedge \bar{\theta}^{Q}\otimes
 e,$ the last  big summation  vanishes
in the formula (\ref{jjj}).  $\langle [\Theta_J
,\Lambda_J]\xi,\xi\rangle$ is expressed by the first two big
summations in Eq. (\ref{jjj}) in the following way:
\begin{equation}\begin{array}{rcl}
 &&\langle [\Theta_J ,\Lambda_J]\xi,\xi\rangle (x)\\
 &=&\sum_{|P|=p,|Q|=q}(\sum_{p\in Q,1\leq p\leq 2n}\frac{\nu_p (x)}{\mu_p (x)}
 +\sum_{p+n\in Q,1\leq p\leq n}\frac{\nu_p (x)}{\mu_p (x)}\\
 &&+\sum_{p\in Q,1\leq p\leq n}\frac{\nu_{p+n} (x)}{\mu_{p+n} (x)} -\sum_{p=1}^{2n}
 \frac{\nu_p (x)}{\mu_p (x)})|\xi_{P,Q}|^2\\
 &\geq &[2(\frac{\nu_1 (x)}{\mu_1 (x)} +\cdots +\frac{\nu_p (x)}{\mu_p (x)})

 -\sum ^{2n}_{j=1}\frac{\nu_j (x)}{\mu_j (x)}]|\xi|^2.\\
 \end{array}\label{213}\end{equation}
Observe that if  the ratio $\nu_j (x)/\mu_j (x)$ vary small when $j$
varies, for example, in the extreme case when all $\nu_j (x)/\mu_j
(x)$ are equal, then $[i\Theta_{h} (E),\Lambda](x)$ is positive when
$p>n.$ This observation tells us that if we choose the K\"ahler
metric $\omega$ properly such that the eigenvalues  of $i\Theta_{h}
(E)$ vary mildly relative to $\omega,$ then we can deduce the
positivity of $[i\Theta_{h} (E),\Lambda].$

Since $B$ is $k$-positive if and only if its curvature
   $i\Theta $ is
  semi-positive  with rank $i\Theta \geq n-k.$
We take a special choice
of the K\"ahler metric of $M$ with $\widetilde{\omega}:=\widetilde{\omega}_I:=i\Theta +\kappa
\omega$ for some positive number $\kappa.$ Then $M$ is a compact hypercomplex manifold and $(M,I)$ is still  K\"ahler manifold with the new K\"ahler form
$\widetilde{\omega}.$ Form now on we consider $M$ as
 a new hypercomplex K\"ahler manifold with K\"ahler metric $\widetilde{\omega}.$
Correspondingly, we have three new nondegenerate 2-froms,
$\widetilde{\omega}_I,\widetilde{\omega}_J,\widetilde{\omega}_K.$
  Let
$\widetilde{\Lambda}_I,\widetilde{\Lambda}_J,\widetilde{\Lambda}_K$
be the associated adjoint operators of multiplication by
$\widetilde{\omega}_I,\widetilde{\omega}_J,\widetilde{\omega}_K.$
 We could get the Bochner-Kodaira-Nakano identities with respect
to $\widetilde{\omega}_I,\widetilde{\omega}_J,\widetilde{\omega}_K$ as in Section 4.
Then the eigenvalues of $\Theta$
relative to $\widetilde{\omega}$ are $\{r_j(x)\}_{1\leq j\leq n}$
with
$$r_j(x)=\frac{\nu_j (x)}{\kappa\mu_j (x)
+\nu_j (x)}=\frac{\nu_j (x)/\mu_j (x)}{\kappa+\nu_j (x)/\mu_j (x)}=
\frac{\frac{\nu_j (x)/\mu_j (x)}{\kappa}}{1+{\frac{\nu_j (x)/\mu_j
(x)}{\kappa}}} =1-\frac{1}{1+{\frac{\nu_j (x)/\mu_j
(x)}{\kappa}}}.$$ Fix a point $x\in M$ and assume that ${\rm
rank}\big(\Theta|_x\big) =2n-s\geq 2n-k.$ Then $0= r_s (x)<
r_{s+1}(x)\leq\cdots\leq r_{2n} (x).$ Thus $r_j (x)=0$ for $j\leq
s.$ If we choose $\kappa\rightarrow 0^+$ then $r_j (x)\rightarrow 1$
for all $s+1\leq j\leq 2n.$ If $p>n+[\frac{k}{2}],$ then $p\geq
n+\frac{k}{2}\geq k\geq s$ and
$$\begin{array}{rcl}
&&\lim_{\kappa\rightarrow 0^+}[2(r_1 (x)+\cdots +r_p (x))
-( r_{1}(x) +\cdots +r_{2n} (x)))\\
&=& 2(p-s)-(2n-s)\\
&\geq & 2(p-(n+\frac{k}{2}))\\
&\geq &2(1+[\frac{k}{2}]-\frac{k}{2})\geq 1.\\
\end{array}$$
Since $M$ is compact, we can use a finite cover by open sets, such
that on each open set, if  $\kappa$ is a sufficiently small positive
number we may have on each open neighborhood  and hence everywhere
on $M$ that
$$\langle [\Theta_J,\widetilde{\Lambda}_J]\xi,\xi\rangle
 \geq  \frac{1}{2}|\xi|^2.
 $$
 Therefore if $p>n+[\frac{k}{2}]$  then $[\Theta_J,\widetilde{\Lambda}_J]$ is positive on $\Omega^{q,p}(B)$ and
 (ii) of Theorem \ref{mainr} follows.
\end{proof}

\begin{corollary}\label{lste}
 Let $B$ be a holomorphic $k$-positive
 line bundle on a compact hyperk\"ahler manifold $M.$  Then

{\rm (i).}~~~ $H^p(M,\Omega^q\otimes B)=0,$\quad for \quad
$p+q>2n+k;$

{\rm (ii).}~~~ $H^p (M,\Omega^q\otimes B)=0,$ \quad for
$p>n+[\frac{k}{2}]$ and any  nonnegative integer $q.$
\end{corollary}

In particular, if $B$ is a positive holomorphic line bundle, we have

{\rm (i)$^{\prime}.$}~~~$H^p(M,\Omega^q\otimes B)=0,$\quad for \quad
$p+q>2n;$

{\rm (ii)$^{\prime}.$}~~~ $H^p (M,\Omega^q\otimes B)=0,$ \quad for
$p>n$ and any nonnegative integer $q.$

In \cite{ya}, we proved that on a compact K\"ahler manifold, any $k$-ample line bundle is $k$-positive.
So Corollary \ref{lste} is also applicable to $k$-ample line bundle. In particular, if $k=0,$ we get Theorem
\ref{maincor} for algebraic hyperk\"ahler manifolds.

\section*{Acknowledgments}  While the first draft of this paper was finished,
I gave a  talk  at the
Institute of Mathematics of
Chinese University of Hong Kong, where Prof. Leung, Naichung Conan pointed out to the author
that partial results of Theorem \ref{vbts} could also be proved by using the holomorphic hard Lefschetz
isomorphism (twisted by a positive line bundle) together with the  Akizuki-Nakano vanishing theorem \cite{an}.
His ideal may also be applicable to other part of our work and
will be considered in another paper. I would like to thank him for his skillful comments.

This work was partly supported by the Fundamental Research Funds for
the Central Universities (No. 09lgpy49) and  NSFC (No. 10801082).

 \end{document}